\documentclass[12pt,a4paper,final]{amsart}
\usepackage{amssymb}
\usepackage{hyperref}
\usepackage[notcite, notref]{showkeys}
\usepackage[utf8]{inputenc}

\theoremstyle{plain}
\newtheorem{theorem}{Theorem}[section]
\newtheorem{proposition}[theorem]{Proposition}

\newtheorem{lemma}[theorem]{Lemma}

\theoremstyle{definition}
\newtheorem{definition}[theorem]{Definition}
\newtheorem{example}[theorem]{Example}
\newtheorem{remark}[theorem]{Remark}
\newtheorem{question}[theorem]{Question}
\newtheorem{conjecture}[theorem]{Conjecture}

\newtheorem{notation}[theorem]{Notation}

\theoremstyle{remark}

\numberwithin{equation}{section}

\newcommand{\N}{\mathbb N}

\newcommand{\R}{\mathbb R}
\newcommand{\C}{\mathbb C}

\newcommand{\fa}{\mathfrak a}

\newcommand{\fc}{\mathfrak c}

\newcommand{\fg}{\mathfrak g}
\newcommand{\fh}{\mathfrak h}

\newcommand{\fk}{\mathfrak k}

\newcommand{\fp}{\mathfrak p}
\newcommand{\fq}{\mathfrak q}

\DeclareMathOperator{\GL}{GL}

\DeclareMathOperator{\Ot}{O}
\DeclareMathOperator{\SO}{SO}
\DeclareMathOperator{\SU}{SU}
\DeclareMathOperator{\Sp}{Sp}
\DeclareMathOperator{\Ut}{U}

\DeclareMathOperator{\Spin}{Spin}

\newcommand{\gl}{\mathfrak{gl}}

\newcommand{\op}{\operatorname}
\newcommand{\Id}{\operatorname{Id}}

\DeclareMathOperator{\diag}{diag}
\DeclareMathOperator{\Span}{Span}
\newcommand{\inner}[2]{\langle {#1},{#2}\rangle }
\newcommand{\innerdots}{\langle \cdot,\cdot \rangle }

\DeclareMathOperator{\End}{End}
\DeclareMathOperator{\Sym}{Sym}

\DeclareMathOperator{\Ad}{Ad}

\DeclareMathOperator{\Cas}{Cas}

\DeclareMathOperator{\Spec}{Spec}

\DeclareMathOperator{\diam}{diam}

\DeclareMathOperator{\dist}{dist}
\newcommand{\HH}{\mathcal H}
\newcommand{\CC}{\mathcal C}
\newcommand{\FF}{\mathcal F}

\title[Diameter and Laplace eigenvalue estimates]{Diameter and Laplace eigenvalue estimates for compact homogeneous Riemannian manifolds}
\author{Emilio~A.~Lauret}
\address{Instituto de Matemática (INMABB), Departamento de Matemática, Universidad Nacional del Sur (UNS)-CONICET, Bahía Blanca, Argentina.}
\email{emilio.lauret@uns.edu.ar}

\thanks{This research was supported by grants from FONCyT (PICT-2018-02073) and SGCYT--UNS}
\date{February 8, 2021}

\begin{document}

\begin{abstract}
Let $G$ be a compact connected Lie group and let $K$ be a closed subgroup of $G$. 
In this paper we study whether the functional $g\mapsto \lambda_1(G/K,g)\operatorname{diam}(G/K,g)^2$ is bounded among $G$-invariant metrics $g$ on $G/K$. 
Eldredge, Gordina, and Saloff-Coste conjectured in 2018 that this assertion holds when $K$ is trivial; the only particular cases known so far are when $G$ is abelian, $\operatorname{SU}(2)$, and $\operatorname{SO}(3)$. 
In this article we prove the existence of the mentioned upper bound for every compact homogeneous space $G/K$ having multiplicity-free isotropy representation. 
\end{abstract}

\maketitle

\tableofcontents
	
\section{Introduction}\label{sec:intro}

A Riemannian manifold is called homogeneous if the action of its isometry group on it is transitive. 
In \cite{Li80}, Peter Li proved for every compact connected homogeneous Riemannian manifold $(M,g)$ that 
\begin{equation}\label{eq1:Li-estimate}
\lambda_1(M,g) \geq \frac{\pi^2/4}{\diam(M,g)^2}.
\end{equation} 
Here, $\diam(M,g)$ stands for the diameter of $(M,g)$ and $\lambda_1(M,g)$ denotes the smallest positive eigenvalue of the Laplace--Beltrami operator associated to $(M,g)$. 
(See \cite{JudgeLyons17} for an improvement of  \eqref{eq1:Li-estimate}.)

The analogous upper bound to \eqref{eq1:Li-estimate} does not exist since 
\begin{equation}
\lambda_1(S^d,g_{\text{round}}) \diam(S^d,g_{\text{round}})^2 = d\pi^2\longrightarrow \infty
\qquad\text{when $d\to\infty$},
\end{equation}
where $g_{\text{round}}$ denotes any round metric on the $d$-sphere $S^d$. 
It is currently unknown whether there exists an upper bound analogous to \eqref{eq1:Li-estimate} when the dimension is fixed. 
Eldredge, Gordina and Saloff-Coste conjectured in \cite{EldredgeGordinaSaloff18} the following. 

\begin{conjecture}\label{conj1:EGS}
For any compact connected Lie group $G$, there is a real number $C$ depending only on $G$ such that $\lambda_1(G,g)\diam(G,g)^2 \leq C$ for every left-invariant metric $g$ on $G$. 
\end{conjecture}

Lie groups endowed with left-invariant metrics form a subclass of homogeneous Riemannian manifolds. 
The aim of this paper is to study the extension of this conjecture to homogeneous Riemannian manifolds. 
It is well known that for every homogeneous Riemannian manifold $(M,g)$ there are Lie groups $K\subset G$ such that $(M,g)$ is isometric to $G/K$ endowed with some $G$-invariant metric.

\begin{question}\label{question}
Given any compact connected Lie group $G$ and any closed subgroup $K$ of $G$ such that the quotient $G/K$ is connected, is there $C>0$ depending only on $(G,K)$ such that $\lambda_1(G/K,g)\diam(G/K,g)^2 \leq C$ for every $G$-invariant metric $g$ on $G/K$?
\end{question}
 
The natural extension of Conjecture~\ref{conj1:EGS} to the homogeneous setting is the following:
\begin{quote}
\em 
Given any compact connected Lie group $G$ and any closed subgroup $K$ of $G$ such that $G/K$ is connected, there is $C=C(G,K)>0$ such that 
\begin{equation}\label{eq1:estimate}
\lambda_1(G/K,g) \leq \frac{C}{\diam(G/K,g)^2}
\end{equation}
for every $G$-invariant metric $g$ on $G/K$. 
\end{quote}

The next result tells us that a positive answer to Question~\ref{question} for $(G,\{e\})$ yields the same positive answer for $(G,K)$ for every closed connected subgroup $K$ of $G$. 
However, its usefulness is very limited at the moment since Question~\ref{question} for $(G,\{e\})$ is known only for a few cases (described below).

\begin{theorem}\label{thm1:comparison}
Let $G$ be any compact connected Lie group.
Assume there is $C>0$ such that $\lambda_1(G,g) \diam(G,g)^2 \leq C$ for every left-invariant metric $g$ on $G$.
Then, for any closed connected subgroup $K$ of $G$ with $G/K$ connected, $\lambda_1(G/K,h) \diam(G/K,h)^2 \leq C$ for every $G$-invariant metric $h$ on $G/K$. 
\end{theorem}

Our approach will extend the tools developed in \cite{Lauret-EGSconj} for the Lie group case to the homogeneous setting.
The main difference can be observed in the number of cases where we have a complete answer to Question~\ref{question} (always affirmatively so far). 

On the one hand, Question~\ref{question} for $(G,\{e\})$ has been answered only for $G$ abelian, $G=\SU(2)$, and $G=\SO(3)$ in \cite{EldredgeGordinaSaloff18} after elaborated proofs (see also \cite[Thm.~1.4]{Lauret-SpecSU(2)}). 
On the other hand, for instance, Question~\ref{question} is trivially affirmative for every isotropy irreducible pair $(G,K)$ since there is a unique $G$-invariant metric on $G/K$ up to positive scaling and the term $\lambda_1(G/K,g) \diam(G/K,g)^2$ is invariant by homotheties. 
Moreover, we are able to answer Question~\ref{question} for a large subclass of homogeneous spaces $G/K$.

\begin{theorem}\label{thm1:multfreeisotropy}
Let $G$ be a compact connected semisimple Lie group and let $K$ be a closed subgroup of $G$ such that $G/K$ is connected. 
Assume that the isotropy representation of $G/K$ is multiplicity free, that is, it decomposes as a direct sum of pairwise non-equivalent irreducible representations. 
Then, there is $C=C(G,K)>0$ such that 
$$
\lambda_1(G/K,g) \leq \frac{C}{\diam(G/K,g)^2}
$$ 
for every $G$-invariant metric $g$ on $G/K$. 
\end{theorem}

The class of compact connected homogeneous spaces $G/K$ with multiplicity-free isotropy representation is quite large. 
It contains for instance generalized flag manifolds (see e.g.\ \cite{Arvanitoyeorgos06}) and many real flag manifolds.
There are classifications and lists for low values of the number of different irreducible components. 
See \cite{WangZiller91}, \cite{DickinsonKerr08}, and \cite{ChenKangLiang16,Nikonorov16} for the one, two, and three-components cases respectively.

Theorem~\ref{thm1:multfreeisotropy} follows immediately from Theorem~\ref{thm5:EGSpartial}. 
This more general result gives a partial (affirmative) answer to Question~\ref{question} in the sense that it ensures the existence of $C>0$ such that $\lambda_1(G/K,g)\diam(G/K,g)^2 \leq C$ for every $g$ in certain subsets of $G$-invariant metrics on $G/K$. 
More precisely, given a decomposition in $\Ad(K)$-invariant subspaces $\fp=\fp_1\oplus\dots\oplus \fp_q$ of the orthogonal complement $\fp$ of $\fk$ in $\fg$ with respect to a bi-invariant inner product $\innerdots_0$ on $\fg$, the subset is given by $G$-invariant metrics associated to the $\Ad(K)$-invariant inner products on $\fp$ given by 
\begin{equation}\label{eq1:diagonal}
x_1\innerdots_0|_{\fp_1\times \fp_1}+\dots+ x_q\innerdots_0|_{\fp_q \times \fp_q},
\end{equation}
for $x_1,\dots,x_q>0$. 
(See Subsection~\ref{subsec:G-invmetrics} for a detailed description of the correspondence between $G$-invariant metrics on $G/K$ and $\Ad(K)$-invariant inner products on $\fp$.)
In particular, Theorem~\ref{thm5:EGSpartial} for $K$ trivial coincides with \cite[Cor.~1.4]{Lauret-EGSconj}. 

\begin{remark}
The main result in \cite{Lauret-EGSconj} concerning Conjecture~\ref{conj1:EGS} has an involved statement giving a partial answer to Question~\ref{question} for $(G,\{e\})$ with a much larger subset of left-invariant metrics on $G$ as explained above. 
By writing $m=\dim G$, the order of the dimension of this set is $m^2$, the same order as the dimension of the full space of left-invariant metrics on $G$ (which is exactly $\frac12 m(m+1)$), while the dimension of the subset discussed above is $m$. 

It is possible to give the analogous result to \cite[Thm.~1.3]{Lauret-EGSconj} in the homogeneous context. 
However, the corresponding statement is again quite involved and it does not worth in the understanding of Conjecture~\ref{conj1:EGS} in the author's opinion. 
\end{remark}

At this point, it is worth mentioning another partial answer to Question~\ref{question} for $(G,\{e\})$.
In \cite{Lauret-natred}, when $G$ is simple, the author proved the existence of $C>0$ satisfying \eqref{eq1:estimate} for every naturally reductive metric $g$ on $G$. 
The class of naturally reductive metrics is considered as a natural extension of symmetric metrics. 

Theorem~\ref{thm5:EGSpartial} can also be applied to some compact homogeneous spaces $G/K$ with non-multiplicity-free isotropy representation. 
For instance, Theorem~\ref{thm5:EGSpartial} answers  Question~\ref{question} (positively) for the case $(G,K)=(\Sp(n+1), \Sp(n))$, in which case $G/K$ is diffeomorphic to the $(4n+3)$-dimensional sphere $S^{4n+3}$ and its isotropy representation decomposes as the standard representation plus three times the trivial representation. 
As a consequence of the classification of homogeneous metrics on the underlying manifold of a simply connected compact symmetric spaces of real rank one given by Ziller~\cite{Ziller82}, we obtain the following result. 

\begin{theorem}\label{thm1:spheres}

For every positive integer $d$, there is $C=C(d)>0$ such that 
$$
\lambda_1(X,g) \leq \frac{C}{\diam(X,g)^2}
$$ 
for every homogeneous metric $g$ on $X$, where $X$ is the underlying differentiable manifold of any compact simply connected Riemannian symmetric space of real rank one (i.e.\ $X=S^d$, $X=P^{d/2}(\C)$ if $d$ is even, $X=P^{d/4}(\mathbb H)$ if $d$ is divisible by $4$, and $X=P^{2}(\mathbb O)$ if $d=16$).
\end{theorem}

\subsection*{Organization}
Section~\ref{sec:preliminaries} recalls the (implicit) description of the spectrum of a compact homogeneous Riemannian manifold. 
It also includes some estimates for the diameter and first Laplace eigenvalue of some $G$-invariant non-Riemannian structures on a homogeneous space. 
Theorem~\ref{thm1:comparison} is proved in Section~\ref{sec:comparison}. 
Section~\ref{sec:diam} and \ref{sec:eigenvalues} establish the estimates for the diameter and the first Laplace eigenvalue respectively, in terms of the numbers $x_1,\dots,x_q$ in \eqref{eq1:diagonal}. 
The last section proves Theorem~\ref{thm5:EGSpartial}, which is the most general result of the article concerning Question~\ref{question}, and also  Theorems~\ref{thm1:multfreeisotropy} and \ref{thm1:spheres} as consequences of it.

\subsection*{Acknowledgments}
The author is grateful for helpful and motivating conversations with 
Andreas Arvanitoyeorgos and Jorge Lauret. 
The author is greatly indebted to the referees for many accurate comments and clarifications that have helped to improve significantly the presentation of the paper.

\section{Preliminaries}\label{sec:preliminaries}

In this section we fix a parameterization of $G$-invariant metrics on a homogeneous space $G/K$. 
Then, we recall the well-known description of the spectrum of the Laplace--Beltrami operator associated to such metric. 
We conclude with a study of the diameter and the first eigenvalue of the Laplacian associated to two left-invariant non-Riemannian structures: sub-Riemannian manifolds and singular Riemannian manifolds.

\begin{remark}\label{rem2:assumption+notation}
Throughout the article, 
\begin{itemize}
\item we assume that $G$ is a compact connected Lie group and $K$ is a closed subgroup of $G$ such that $G/K$ is connected;

\item we denote by $\fg$ and $\fk$ the Lie algebras of $G$ and $K$ respectively;

\item we write $m=\dim \fg$ and $n=\dim G/K=\dim \fg-\dim \fk$, thus $\dim \fk=m-n$;

\item we fix an $\Ad(G)$-invariant inner product $\innerdots_0$ on $\fg$, which exists because $G$ is compact;

\item let $\fp$ be the orthogonal complement of $\fk$ in $\fg$ with respect to $\innerdots_0$.
\end{itemize}
\end{remark}

\subsection{Invariant metrics on a homogeneous space}\label{subsec:G-invmetrics}
Each $X\in\fg$ defines a vector field on $G/K$ given by $X_{aK}=\left.\frac{d}{dt}\right|_{t=0} \exp(tX)aK$ for $a\in G$. 
The map $X\mapsto X_{eK}$ identifies $\fp$ with the tangent space $T_{eK}G/K$.

The subspace $\fp$ of $\fg$ is invariant by $\Ad(a)$ for all $a\in K$, thus it has the structure of a $K$-module. 
The isotropy representation of $G/K$ coincides with this representation, $\Ad:K\to \GL(\fp)$.

For $a\in G$, we define $\tau_a:G/K\to G/K$ given by $\tau_a(bK)= abK$ for $b\in G$. 
A Riemannian metric $g$ on $G/K$ is called \emph{$G$-invariant} if $g_{eK}(\cdot,\cdot)=g_{aK}(d\tau_a\cdot ,d\tau_a\cdot )$ for all $a\in G$; in other words, $\tau_a$ is an isometry for all $a\in G$. 
In this case, it follows that $(G/K,g)$ is a \emph{homogeneous Riemannian manifold}, that is, a Riemannian manifold whose isometry group acts transitively on it.
Every compact homogeneous Riemannian manifold is isometric to some $(G/K,g)$ as above. 
We denote by $\mathcal M(G,K)$ the space of $G$-invariant metrics on $G/K$.

It is well known that there is a one-to-one correspondence between the space of $G$-invariant metrics on $G/K$ and the set of $\Ad(K)$-invariant inner products on $\fp$.
Given any $\Ad(K)$-invariant inner product $\innerdots$ on $\fp$, we define the Riemannian metric on $G/K$ by $g_{aK}(\cdot,\cdot)=\inner{d\tau_{a^{-1}}\cdot}{ d\tau_{a^{-1}}\cdot}$ for all $a\in G$ via the identification $T_{eK}G/K\equiv \fp$. 
The metric is well defined since the inner products $g_{aK}$ and $g_{bK}$ for $aK=bK$ coincide because $\innerdots$ is $\Ad(K)$-invariant. 

To parametrize $\mathcal M(G,K)$ we need to consider the decomposition of the isotropy representation as irreducible components. 
There are pairwise non-equivalent irreducible (real) representations $W_1,\dots,W_r$ of $K$ and $q_1,\dots,q_r\in\N$ such that 
\begin{equation}\label{eq2:isotropydecomposition}
\fp\simeq q_1 W_1\oplus\dots\oplus q_rW_r
\qquad\text{as $K$-modules,}
\end{equation}
where $q_jW_j$ denotes $q_j$-copies of $W_j$. 
As a consequence, there are subspaces $\fc_1,\dots,\fc_r$ of $\fp$ such that
\begin{equation}\label{eq2:isotypical}
\fp=\fc_1\oplus\dots \oplus \fc_r
\qquad\text{and}\qquad
\text{$\fc_j\simeq q_jW_j$ as $K$-modules}.
\end{equation}
This decomposition is unique up to order. 
The subspace $\fc_j$ of $\fp$ is called the \emph{isotypical} component of type $W_j$. 
One can check that $\fc_1,\dots,\fc_r$ are mutually orthogonal with respect to any $\Ad(K)$-invariant inner product on $\fp$. 

Of course, $\innerdots_0|_{\fp\times \fp}$ is $\Ad(K)$-invariant. 
Every $\Ad(K)$-invariant inner product on $\fp$ is of the form
\begin{equation}\label{eq2:inner_Phi}
\inner{X}{Y}_{\Phi} = \inner{\Phi(X)}{Y}_0
\end{equation}
for $X,Y\in\fp$, for some  $\Phi$ in 
\begin{equation}
\End_K(\fp):=\{\Phi\in\End(\fp): \Phi \Ad(a)=\Ad(a)\Phi\quad\text{for all }a\in K\}
\end{equation}
symmetric and positive definite with respect to $\innerdots_0$. 
We denote by $\Sym_K^+(\fp)$ the set of such elements. 
For $\Phi \in \Sym_K^+(\fp)$, we denote by $g_{\Phi}$ the induced $G$-invariant metric on $G/K$.

It is important to mention that
\begin{equation}
\begin{aligned}
\End_K(\fp)
=\bigoplus_{j=1}^r \End_K(\fc_j) 
\simeq \bigoplus_{j=1}^r \End_K(q_jW_j)
\simeq \bigoplus_{j=1}^r \gl_{q_j}(\End_K(W_j))
\end{aligned}
\end{equation} 
with $\mathbb F_j :=\End_K(W_j)\simeq \R, \C, \mathbb H$ according $W_j$ is of real, complex, or quaternionic type. 
Consequently, $\mathcal M(G,K)$ is in correspondence with 
\begin{equation}
\bigoplus_{j=1}^r\; \{S\in \gl_{q_j}(\mathbb F_j) : \overline{S^t}=S\text{ and } S>0\},
\end{equation}
where the conjugation is induced by the (restriction of the) standard conjugation on the Hamiltonian quaternions $\mathbb H$.

\begin{example}
When $K$ is trivial, one has that $\fp=\fg$. 
Elements in $\mathcal M(G,K)$ are called left-invariant metrics on $G$. 
Its isotropy representation $\fp=\fg$ decomposes as $m$-times the trivial representation, which is of real type. 
Consequently, $\mathcal M(G,K)$ is in correspondence with the space of $m\times m$ (real) positive definite symmetric matrices. 
\end{example}

\begin{example}\label{ex2:multiplicityfree}
We assume that the isotropy representation is \emph{multiplicity free}, that is, $q_j=1$ for all $j$. 
If $\Phi\in \Sym_K^+(\fp)$, since $\Phi(\fc_j)=\fc_j$ and $\fc_j$ is irreducible for all $j$, then $\Phi$ acts by a positive scalar on $\fc_j$. 
That is, there are positive real numbers $\sigma_1,\dots, \sigma_r$ such that 
\begin{equation}
\Phi=\bigoplus_{j=1}^r \, 
\sigma_j^{-2}\,\Id_{\fc_j}. 
\end{equation}
\end{example}

We now return to the general case. We set $q=q_1+\dots+q_r$, which is the number of irreducible components in the isotropy representation.

\begin{notation}\label{not2:diagonaldecomposition}
We fix $\Phi\in \Sym_K^+(\fp)$. 
Since the eigenspaces of $\Phi$ are $\Ad(K)$-invariants, there are non-trivial $\Ad(K)$-invariant subspaces $\fp_{1}, \dots,\fp_{q}$ of $\fp$ and positive real numbers $\sigma_1(\Phi)\geq \dots\geq \sigma_q(\Phi)$ such that $\Phi|_{\fp_{i}}=\sigma_i(\Phi)^{-2}\, \Id_{\fp_{i}}$ for all $1\leq i\leq q$. 
We call to the decomposition
\begin{equation}\label{eq2:diagonaldecomposition}
\fp = \fp_{1}\oplus \dots\oplus \fp_{q}
\end{equation}
a \emph{diagonal decomposition of $\Phi$}. 
This decomposition is unique only if $\sigma_1(\Phi)> \dots> \sigma_q(\Phi)$. 
For each $1\leq j\leq r$, $\fp_i\simeq W_j$ as $K$-modules for exactly $q_j$ indexes $i\in\{1,\dots,q\}$. 
\end{notation}

\begin{remark}\label{rem2:basis}
Given any diagonal decomposition $\fp = \fp_{1}\oplus \dots\oplus \fp_{q}$ of $\Phi$, the subspaces $\fp_{1},\dots,\fp_{q}$ are mutually orthogonal with respect to $\innerdots_{\Phi}$ and $\innerdots_0$. 
Consequently, there exists a $\innerdots_0$-orthonormal basis $\{X_{i,k}: 1\leq i\leq q,\, 1\leq k\leq \dim \fp_i\}$ of $\fp$ such that $\Span_\R\{X_{i,k}: 1\leq k\leq \dim \fp_i\}= \fp_{i}$ for all $i$ and
\begin{equation}
\inner{X_{i,k}} {X_{i',k'}}_{\Phi} = 
\begin{cases}
\sigma_{i}(\Phi)^{-2} \quad&\text{if }i=i',\, k=k',\\
0\quad&\text{otherwise}. 
\end{cases}
\end{equation}
\end{remark}

From now on, for $\Phi,\Psi\in \Sym_K^+(\fp)$, we write $\Phi\geq \Psi$ if $\Phi-\Psi$ is positive semi-definite with respect to $\innerdots_0$, that is, $\inner{\Phi(X)}{X}_0 \geq \inner{\Psi(X)}{X}_0$ for all $X\in \fp$. 
The next result follows immediately from \eqref{eq2:inner_Phi}.

\begin{lemma}\label{lem2:A^tAleqB^tB}
Let $\Phi,\Psi\in \Sym_K^+(\fp)$ such that $\Phi\geq \Psi$. 
Then $g_\Phi\geq g_\Psi$. 
\end{lemma}

\subsection{Spectra of invariant metrics on a homogeneous space}
\label{subsec:spectraleft-invmetrics}

Let $\pi:G\to \GL(V_\pi)$ be a finite dimensional unitary representation of $G$.
As an abuse of notation, we denote again by $\pi$ to its differential, which is a representation of $\mathfrak g$, and the corresponding representation on the universal enveloping algebra $\mathcal U(\fg)$. 
 
Let $\langle \cdot,\cdot\rangle_\pi$ denote the inner product on $V_\pi$. 
Since $\pi(a):V_\pi\to V_\pi$ is unitary for every $a\in G$, $\pi(X)$ is skew-hermitian for every $X\in\mathfrak g$, i.e.\ $\langle \pi(X)v,w\rangle_\pi = -\langle v,\pi(X)w\rangle_\pi$ for all $v,w\in V_\pi$. 
Hence $\pi(-X^2)=-\pi(X)\circ\pi(X)$ is self-adjoint and positive semi-definite.
It follows that $\pi(-C_\Phi)$ is self-adjoint and positive semi-definite.

We denote by $\widehat G$ the unitary dual of $G$, that is, the collection of equivalence classes of irreducible unitary representations of $G$. 
For each $(\pi,V_\pi)\in\widehat G$, one has the embedding
\begin{align} \label{eq2:PeterWeyl-embedding}
	V_\pi^K\otimes V_\pi^* &\longrightarrow C^\infty(G/K)\equiv  C^\infty(G)^K ,\\
	v\otimes\varphi &\longmapsto 
	\big(x\mapsto f_{v\otimes\varphi}(x):= \varphi(\pi(x) v) \big).
	\notag
\end{align}
Here, $C^\infty(G)^K= \{f\in C^\infty(G): f(xa)=f(x) \text{ for all $x\in G$, $a\in K$}\}$ and $V_\pi^K=\{v\in V_\pi: \pi(a)v=v\text{ for all }a\in K \}$. 
Note that 
\begin{equation}\label{eq2:pi(C_k)v=0}
\pi(X)v=0
\qquad\text{for all } v\in V_\pi^K,\; X\in\fk. 
\end{equation}
It is a simple matter to check that $V_\pi^K$ is invariant by $\pi(X)$ for every $X\in \fp$.

We fix $\Phi\in\Sym_K^+(\fp)$.
Given any basis $\{Y_1,\dots,Y_n\}$ of $\fp$, we set $S=(\inner{Y_i}{Y_j}_\Phi)_{i,j}$ and $T=S^{-1}=(t_{i,j})_{i,j}$. 
Both are $n\times n$ positive definite symmetric matrices. 
One can check that the element $\sum_{i,j=1}^n t_{i,j}\; Y_{i}\,Y_{j}$ in $\mathcal U(\fg)$ does not depend on the basis chosen. 
We thus set 
\begin{equation}\label{eq2:C_Phi}
C_\Phi := \sum_{i,j=1}^n t_{i,j}\; Y_{i}\,Y_{j} \in \mathcal U(\fg).
\end{equation}
In particular, if $\{X_{i,k}\}$ is an orthonormal basis of $\fp$ with respect to $\innerdots_0$ that respects a diagonal decomposition $\fp=\fp_1\oplus\dots\oplus \fp_q$ of $\Phi$ as in Remark~\ref{rem2:basis}, then 
\begin{equation}\label{eq2:C_Phi-diagonal}
C_\Phi=\sum_{i=1}^q \sum_{k=1}^{\dim \fp_i} \sigma_i(\Phi)^{2}\, X_{i,k}^2. 
\end{equation}

We denote by $\Delta_\Phi$ the Laplace--Beltrami operator associated to the Riemannian manifold $(G/K,g_\Phi)$. 
One has that (see for instance \cite[Thm.~1 and (3.1)]{MutoUrakawa80}; see also \cite[\S2]{BLPhomospheres}) 
\begin{equation}\label{eq2:Laplacian}
\Delta_\Phi\cdot f_{v\otimes\varphi} = f_{(\pi(-C_\Phi)v)\otimes\varphi}
\end{equation}
for $v\in V_\pi^K$ and $\varphi\in V_\pi^*$. 
Note that $\pi(-C_\Phi)(V_\pi^K)\subset V_\pi^K$ since $V_\pi^K$ is invariant by $\pi(X)$ for all $X\in\fp$.  

Suppose that $v\in V_\pi^K$ is an eigenvector of the finite-dimensional linear operator $\pi(-C_\Phi)|_{V_\pi^K}:V_\pi^K \to V_\pi^K$ associated to the eigenvalue $\lambda$, i.e.\ $\pi(-C_\Phi)v=\lambda v$. 	
Then, 
\begin{equation}
\Delta_\Phi \cdot f_{v\otimes\varphi} = f_{(\pi(-C_\Phi)v)\otimes \varphi}
= f_{(\lambda v)\otimes \varphi}=\lambda\, f_{v\otimes\varphi},
\end{equation}
that is, $f_{v\otimes \varphi}$ is an eigenfunction of $\Delta_\Phi$ with eigenvalue $\lambda$, for any $\varphi\in V_\pi^*$.

The Hilbert space $L^2(G/K)\equiv L^2(G)^K:= \{f\in L^2(G): f(xa)=f(x)\text{ for all }a\in K\}$ has a canonical structure of $G$-module, namely, the \emph{left-regular representation} given by $(a\cdot f)(x)=f(a^{-1}x)$ for $a,x\in G$ and $f\in L^2(G)^K$. 
This representation is unitary. 
The Peter-Weyl Theorem (see e.g.\ \cite[Thm.~1.3]{Takeuchi}) ensures that this representation decomposes as
\begin{equation}\label{eq2:PeterWeyl}
L^2(G)\simeq \bigoplus_{\pi\in \widehat G} V_\pi^K \otimes V_\pi^*,
\end{equation}
where the embedding of $V_\pi^K\otimes V_\pi^*$ in $L^2(G/K)$ is as in \eqref{eq2:PeterWeyl-embedding}.
The action of an element $a\in G$ on $V_\pi^K \otimes V_\pi^*$ is given by $a\cdot (v\otimes\varphi )= v\otimes (\pi^*(a)\varphi)$ since 
\begin{equation}\label{eq:GactiononL^2(G)}
(a\cdot f_{v\otimes \varphi})(x) = f_{v\otimes \varphi}(a^{-1}x) = \varphi(\pi(a^{-1})\pi(x) v) = (\pi^*(a) \varphi)(\pi(x)v) = f_{v\otimes (\pi^*(a)\varphi)}(x).
\end{equation}

We note that the Hilbert sum in \eqref{eq2:PeterWeyl} is restricted to the set \emph{spherical representations} associated to $(G,K)$, namely,
\begin{equation}
\widehat G_K:=\{(\pi,V_\pi)\in\widehat G:V_\pi^K\neq0\}. 
\end{equation}

By the orthogonal relations (see for instance \cite[Thm.~1.4]{Takeuchi}), it follows that 
\begin{equation}\label{eq2:basisL^2(G)}
\bigcup_{\pi\in\widehat G_K}\{f_{v_i\otimes \varphi_j}: 1\leq i\leq d_\pi^K,\; 1\leq j\leq d_\pi\}
\end{equation}
is an orthonormal basis of $L^2(G/K)$, where for each $\pi\in\widehat G_K$, 
\begin{itemize}
\item $d_\pi=\dim V_\pi=\dim V_\pi^*$,

\item $d_\pi^K=\dim V_\pi^K$,

\item $\{v_1,\dots,v_{d_\pi^K}\}$ is any orthonormal basis of $V_\pi^K$, and

\item $\{\varphi_1,\dots,\varphi_{d_\pi}\}$ is any orthonormal basis of $V_\pi^*$.
\end{itemize}

For each $\pi\in \widehat G_K$ non-trivial, we take an orthonormal eigenbasis $\{v_1,\dots,v_{d_\pi^K}\}$ of $\pi(-C_\Phi)|_{V_\pi^K}$, i.e.\ 
$ \pi(-C_\Phi)v_i=\lambda_i^{\pi,\Phi} \, v_i$ 
for some $\lambda_i^{\pi,\Phi}>0$. 
We thus obtain that the basis of $L^2(G/K)$ in \eqref{eq2:basisL^2(G)} contains only eigenfunctions of $\Delta_\Phi$.
Hence,  
\begin{equation}\label{eq2:spec_A}
\Spec(G/K,g_\Phi):=
\Spec (\Delta_\Phi) = \bigcup_{\pi\in\widehat G_K} \big\{\!\big\{
\underbrace{\lambda_i^{\pi,\Phi},\dots, \lambda_i^{\pi,\Phi}}_{d_\pi\text{-times}}:1\leq i\leq d_\pi^K 
\big\}\!\big\}.
\end{equation}
(Here, the double curly brackets is to emphasize that the spectrum is a multiset and not a set.)
The multiplicity $d_\pi$ for each $\lambda_i^{\pi,\Phi}$ above comes from the following fact: $f_{v_i\otimes \varphi_j}$ is an eigenfunction of $\Delta_\Phi$ with eigenvalue $\lambda_i^{\pi,\Phi}$ for every $1\leq j\leq d_\pi$. 

For $T:W\to W$ a linear transformation of a finite-dimensional complex vector space $W$, we denote by $\lambda_{\min}(T)$ its smallest eigenvalue. 
The expression \eqref{eq2:spec_A} yields 
\begin{equation}\label{eq2:lambda1(G,g_A)}
\lambda_1(G/K,g_\Phi) = \min \left\{ \lambda_{\min}(\pi(-C_\Phi)|_{V_\pi^K}): \pi\in\widehat G_K,\, \pi\not\simeq 1_G \right\}. 
\end{equation}

\begin{remark}\label{rem2:C_I}
The case $\Phi=\Id_\fp$ is very particular since $\Cas_\fg:=C_{\Id_\fp} + \Cas_\fk$ lies in the center of $\mathcal U(\mathfrak g)$, where $\Cas_\fk= \sum_{i=1}^{\dim \fk}X_i^2$ for any orthonormal basis $\{X_i\}$ of $\fk$ with respect to $\innerdots_0$. 
(For instance, when $\mathfrak g$ is semisimple and $\innerdots_0$ is minus the Killing form, then $\Cas_\fg$ is the \emph{Casimir element} of $\fg$.)
Thus, for any $\pi\in\widehat G_K$, $\pi(-\Cas_\fg)$ commutes with $\pi(a)$ for every $a\in G$, and then Schur's Lemma yields that $\pi(-\Cas_\fg)$ acts by an scalar on $V_\pi$.
By denoting this scalar by $\lambda^\pi$, i.e.\ $\pi(-\Cas_\fg) =  \lambda^\pi\, \Id_{V_\pi}$, since $\pi(\Cas_\fk)v=0$ for all $v\in V_\pi^K$ by \eqref{eq2:pi(C_k)v=0}, we have that 
\begin{equation}\label{eq:spec_I}
\Spec(G/K,g_{\Id_{\fp}})=\Spec (\Delta_{\Id_{\fp}}) = \bigcup_{\pi\in\widehat G_K} \big\{\!\big\{
\underbrace{\lambda^{\pi},\dots, \lambda^{\pi}}_{d_\pi^K\times d_\pi\text{-times}} 
\big\}\!\big\}.
\end{equation}
\end{remark}

\subsection{Diameter of left-invariant non-Riemannian structures} \label{subsec:diam-non-Riemannian}

A \emph{sub-Riemannian manifold} is a triple $(M,\mathcal D,g)$, where $\mathcal D$ is a subbundle of $TM$ and $g=(g_p)_{p\in M}$ denotes a family of inner product on $\mathcal D$ which smoothly vary with the base point
(see \cite{Montgomery-tour} for a general reference).
A smooth curve $\gamma$ on $(M,\mathcal D,g)$ is called \emph{horizontal} if $\gamma'(t)\in \mathcal D_{\gamma(t)}$ for all $t$. 
The length of a horizontal curve $\gamma:[a,b]\to M$ is equal to $\op{lenght}_{(M,\mathcal D,g)}(\gamma):=\int_a^b g_{\gamma(t)}({\gamma'(t)},{\gamma'(t)})^{1/2} \,dt$. 
The sub-distance between two points $p,q\in M$ is defined as the infimum of $\op{lenght}_{(M,\mathcal D,g)}(\gamma)$ over all horizontal curves $\gamma$ on $M$ connecting $p$ and $q$. 
The corresponding diameter, $\diam(M,\mathcal D,g)$, is given by the supremum of the distances between two points in $M$. 
Consequently, if $M$ is compact, the diameter is $\infty$ if two points in $M$ cannot be joined by a horizontal smooth curve.

The following lemma is clear. 

\begin{lemma}\label{lem2:sub-diam-restriccionRiemanniano}
Let $\mathcal D$ be a subbundle on $M$. 
If $g$ is a Riemannian metric on $M$, then the sub-Riemannian metric $h$ on $(M,\mathcal D)$ given by the restriction of $g$ on $\mathcal D$ (i.e.\ $h_p=g_p|_{\mathcal D_p}$ for all $p\in M$) satisfies 
\begin{equation*}
\diam(M,g)\leq \diam(M,\mathcal D,h). 
\end{equation*}
\end{lemma}

We say that a subbundle $\mathcal D$ satisfies the \emph{bracket-generating condition} (also known as the \emph{Hörmander condition}) if the Lie algebra generated by vector fields in $\mathcal D$ spans at every point the tangent space of $M$. 
For such a $\mathcal D$, provided $M$ is compact, the Chow--Rashevskii Theorem ensures that $\diam (M,\mathcal D,h)<\infty$, for any sub-metric $h$.
In particular, any two points in $M$ can be joined by a horizontal curve.

In what follows we concentrate on $G$-invariant sub-Riemannian structures on $G/K$. 
The analogous treatment when $K=\{e\}$ can be found in \cite[Subsec.~2.3]{Lauret-EGSconj}. 
Given any $\Ad(K)$-invariant subspace $\HH$ of $\fp$ and any $\Ad(K)$-invariant inner product $b(\cdot,\cdot)$ on $\HH$, we associate the $G$-invariant sub-Riemannian structure $(\mathcal D,g)$ on $G/K$ given by 
\begin{align}\label{eq2:left-inv-sub}
\mathcal D&=\bigcup_{a\in G}dL_a(\HH), &
	g_{aK}\big(dL_a(X),dL_a(Y)\big) &= b({X},{Y}) ,
\end{align}
for all $X,Y\in\HH$ and $a\in G$. 
Here, $L_a:G/K\to G/K$ is given by $L_a(x)=axK$ and $\HH$ is seen as a subspace of $T_{eK} G/K\equiv \fp$. 
We next see that $\mathcal D$ and $g_{aK}$ are well defined. 
Suppose that $aK=bK$, thus $b=ac$ for some $c\in K$. 
We have that $L_{b}=L_{ac}:G/K\to G/K$ is given by $xK\mapsto bxK=acxK=a(cxc^{-1})K$, thus $dL_{b}(X)=dL_{ac}(X)= dL_{a}(\Ad(c)\cdot X)$ for all $X\in\fp$. 
Since $\HH$ is $\Ad(K)$-invariant, $dL_{a}(\HH)=dL_{b}(\HH)$ and consequently, $\mathcal D$ is well defined. 
Similarly, since $b(\cdot,\cdot)$ is $\Ad(K)$-invariant, we have for any $X,Y\in\HH$ that
\begin{align*}
g_{bK} \big(dL_b(X),dL_b(Y)\big) 
&=g_{bK}\big(dL_{ac}(X),dL_{ac}(Y)\big) \\
&=g_{bK}\big(dL_{a} (\Ad(c)\cdot X),dL_{a} (\Ad(c)\cdot Y)\big) \\
&= b(\Ad(c)\cdot {X},\Ad(c)\cdot {Y})=b(X,Y) \\
&= g_{aK}\big(dL_{a}(X),dL_{a}(Y)\big), 
\end{align*}
which shows that $g$ is also well defined. 

We will denote the corresponding sub-Riemannian manifold by $(G/K,\HH,g)$ and, as in the Riemannian case, $g$ will be identified with the inner product $g_{eK}=b$ on $\HH$. 

It is a simple matter to check that the map $\HH\mapsto \fk\oplus \HH$ gives a correspondence between $\Ad(K)$-invariant subspaces of $\fp$ and $\Ad(K)$-invariant subspaces in $\fg$ that contains $\fk$.

\begin{definition}\label{def2:bracket-generating}
An $\Ad(K)$-invariant subset $S$ of $\mathfrak \fp$ is called \emph{bracket generating} if the smallest Lie subalgebra of $\fg$ containing $\fk\cup S$ is $\fg$. 
\end{definition}

\begin{example}
Suppose that $G/K$ is isotropy irreducible, for instance, any irreducible symmetric space. 
Since $\{\Ad(a)(X):a\in K\}=\fp$ for every non-zero vector $X\in \fp$, $\fp$ is the only non-trivial $\Ad(K)$-invariant subspace of $\fp$. 
We conclude that every $G$-invariant sub-Riemannian structures on $G/K$ is Riemannian. 
\end{example}

\begin{example}
Suppose that $K=\{e\}$. 
We have that a subset $S$ of $\fg$ is bracket generating if and only if the subspace of $\fg$ $\R$-spanned by the elements $X_1$, $[X_1,X_2]$, $[X_1,[X_2,X_3]]$, and so on, for $X_i\in S$, $i=1,2,\dots$, is precisely $\fg$. 
\end{example}

The next theorem follows immediately from the Chow--Rashevskii Theorem. 
Since we will encounter the situation of the theorem many times in the course of this paper, we state it here.

\begin{theorem}\label{thm2:Chow}
Let $\HH$ be an $\Ad(K)$-invariant subspace of $\fp$.
If $\HH$ is bracket-generating, then $\diam(G/K,\HH, g)<\infty$ for any $\Ad(K)$-invariant inner product $g$ on $\HH$. 
\end{theorem}

For a general treatment of sub-Riemannian geometry we refer the reader to \cite{AgrachevBarilariBoscain-book}, \cite{LeDonne-lecturenotes}, and \cite{Montgomery-tour}. 
In the present article we will only use the few facts just reviewed.

\bigskip

We now introduce the second non-Riemannian structure. 
Given $g=(g_p)_{p\in M}$ such that $g_p$ is a positive semi-definite symmetric bilinear form on $T_pM$ at each point $p\in M$ varying smoothly, $(M,g)$ is called a \emph{singular Riemannian manifold}. 
See \cite{Kupeli} for the general theory on a more general context: \emph{singular pseudo-Riemannian manifolds} (i.e.\ $g_p$ is any symmetric bilinear form on $T_pM$).

The length of a smooth curve $\gamma:[a,b]\to M$ is equal to $\int_a^b g_{\gamma(t)}({\gamma'(t)},{\gamma'(t)})^{1/2} \,dt$. 
The \emph{singular distance} between two points $p,q\in M$ is defined as the infimum of the lengths over all smooth curves $\gamma$ on $M$ connecting $p$ and $q$. 
The corresponding diameter, $\diam(M,g)$, is given by the supremum of the distances between two points in $M$. 

\begin{remark}\label{rem2:pseudo-distance}
The corresponding singular distance $\dist_{(M,g)}$ of $(M,g)$ is a pseudo-distance in the sense of \cite[Def.~1.1.4]{BuragoBuragoIvanov-book}, that is, it satisfies all the properties of a distance except the requirement that $\dist_{(M,g)}(p,q)=0$ implies $p=q$. 
Moreover, the singular diameter of a non-trivial singular Riemannian manifold might be zero (see \cite[Ex.~2.21]{Lauret-EGSconj}). 
By identifying points in $M$ with zero distance in the pseudo-metric space $(M,\dist_{(M,g)})$, we obtain a metric space that we denote by $(M/\dist_{(M,g)} ,\hat \dist_{(M,g)})$ (see for instance \cite[Prop.~1.1.5]{BuragoBuragoIvanov-book}).
\end{remark}

\begin{notation}\label{not2:symmetricbilinearform}
Given any (real) symmetric bilinear form $b$ on $\fg$ and any subspace $\fa$ of $\fg$, let us denote by $b|_{\fa}$ the restriction of $b$ on $\fa$, that is, $b|_{\fa}(X,Y)= b(X,Y)$ for all $X,Y\in\fa$. 
Furthermore, when $b$ is non-degenerate, let $b|_{\fa}^*$ denote the symmetric bilinear form on $\fg$ given by $b|_{\fa}^*(X_1+X_2,Y_1+Y_2)=b(X_1,Y_1)$ for all $X_1,Y_1\in\fa$ and $X_2,Y_2\in\fa^{\perp_b}:=\{X\in\fa:b(X,Y)=0\text{ for all }Y\in\fa\}$. 
Note that if $b$ is positive definite, then $b|_{\fa}$ is positive definite and $b|_{\fa}^*$ is positive semi-definite. 
\end{notation}

The next result is analogous to Lemma~\ref{lem2:sub-diam-restriccionRiemanniano}. 
The proof is again straightforward.

\begin{lemma}\label{lem2:sing-diam-restriccionRiemanniano}
Let $\mathcal D$ be a subbundle on $M$. 
If $g$ is a Riemannian metric on $M$, the singular  Riemannian metric $h$ given by $h_p=g_p|_{\mathcal D_p}^*$ for all $p\in M$ satisfies 
\begin{equation*}
	\diam(M,g)\geq \diam(M,h). 
\end{equation*}
\end{lemma}

We next focus on \emph{$G$-invariant} singular Riemannian structures on a homogeneous space $G/K$. 
Let $b$ be an $\Ad(K)$-invariant positive semi-definite symmetric bilinear form on $\fp$. 
We associate to $b$ the singular Riemannian metric $g$ on $G$ given by
\begin{align}
g_a\big(dL_a(X),dL_a(Y)\big) &= b(X,Y) ,
\end{align}
for all $X,Y\in T_{eK}G/K\equiv \fp$ and $a\in G$. 
One can see that the singular metric $g$ is well defined because $b(\cdot,\cdot)$ is $\Ad(K)$-invariant. 
Similarly as above, we will identify $g$ with the symmetric bilinear form $g_e=b$ on $\fp$. 

\begin{remark}\label{rem2:radical}
The \emph{radical} of a symmetric bilinear form $b$ on $\fp$ is given by 
$$
\op{rad}(b):=\{X\in \fp: b(X,Y)=0\text{ for all }Y\in\fp\}.
$$
If $b$ is non-trivial and $\fa$ is any complement of $\op{rad}(b)$ in $\fp$, then $b|_{\fa}$ is non-degenerate.
\end{remark}

\begin{lemma}\label{lem2:diam(G/d)=diam(G/H)}
Let $H$ be a closed subgroup of $G$ containing $K$ and let $\fq$ denote the orthogonal complement of $\fh$ in $\fg$ with respect to $\innerdots_0$ (so $\fk\subset\fh$ and $\fp\supset\fq$). 
Let $h$ be an $\Ad(H)$-invariant positive semi-definite symmetric bilinear form on $\fp$ with $\op{rad}(h)=\fh\cap \fp$. 
Then, the metric space $(G/\dist_{(G/K,h)}, \hat\dist_{(G/K,h)})$ is isometric (as metric spaces) to the metric space corresponding to the homogeneous Riemannian manifold $(G/H,h|_{\fq})$.
In particular, 
\begin{equation*}
\diam(G/K,h) = \diam(G/H,h|_{\fq}). 
\end{equation*}
\end{lemma}

The proof is left to the reader. 
See \cite[Lem.~2.20]{Lauret-EGSconj} for further details when $K$ is trivial.

\begin{proposition}\label{prop2:diam-pseudo-Hclosed}
Let $h$ be an $\Ad(K)$-invariant positive semi-definite symmetric bilinear form on $\fp$. 
Assume there is a proper subalgebra $\fh$ of $\fg$ containing $\fk\oplus \op{rad}(h)$ such that the associated connected subgroup $H$ of $G$ is closed in $G$. 
Then $\diam(G/K,h) >0$.
\end{proposition}

\begin{proof}
Let $\fq$ be the orthogonal complement subspace of $\fh$ in $\fg$ with respect to $\innerdots_0$, thus $\fq\subset \fp$. 
There is $t>0$ sufficiently small such that $h(X,X)\geq t\, g_0(X,X)$ for all $X\in \fq$. 
We have that $\diam(G/K,h)\geq \diam(G/K,t\,g_0|_{\fq}^*)$ by monotonicity (see e.g.\ \cite[2.17]{Lauret-EGSconj}). 
Now, Lemma~\ref{lem2:diam(G/d)=diam(G/H)} yields $\diam(G/K, t\,g_0|_{\fq}^*)= \diam(G/H,(t\,g_0|_{\fq}^*)|_{\fq} ) = \diam(G/H,t\,g_0|_{\fq})$, which is clearly positive, and the proof is complete. 
\end{proof}

\begin{remark}
The hypothesis of the existence of the subalgebra $\fh$ in  Proposition~\ref{prop2:diam-pseudo-Hclosed} cannot be replaced by assuming that $\op{rad}(h)$ is not bracket generating, unless $G$ is assumed semisimple. 
See \cite[Rem.~2.23]{Lauret-EGSconj} for an explicit counterexample. 
\end{remark}

\subsection{Spectra of left-invariant non-Riemannian structures}
\label{subsec:spec-non-Riemannian}
An element $X\in \fp$ induces a vector field on $G/K$ given as follows: for $f\in C^\infty(G/K)$ and $x\in G$, 
\begin{equation}
(X\cdot f)(xK) = \left.\frac{d}{dt}\right|_{0} f(x\exp(tX)K).
\end{equation}
Note that $X$ commutes with the left-regular representation of $G$. 

Let $\HH$ be an $\Ad(K)$-invariant subspace of $\fp$ and $h$ an $\Ad(K)$-invariant inner product on it. 
The \emph{sub-Laplace operator} (or \emph{sub-Laplacian}) associated to the sub-Riemannian manifold $(G/K,\HH,h)$ (introduced in \eqref{eq2:left-inv-sub}) is the (positive semi-definite self-adjoint) differential operator on $C^\infty(G/K)$ given by 
\begin{equation}
\Delta_{(\HH,h)}(f)= -\sum_{j=1}^l Y_j^2\cdot f,
\end{equation}
where $\{Y_1,\dots,Y_l\}$ is any orthonormal basis of $\HH$ with respect to the inner product $h$.
We set $C_{(\HH,h)}=\sum_{j=1}^l Y_j^2 \in\mathcal U(\fg)$.
For $\pi\in\widehat G$, $v\in V_\pi^K$, $\varphi\in V_\pi^*$, and $f_{v\otimes\varphi}\in C^\infty(G)$ given as in \eqref{eq2:PeterWeyl-embedding}, one has that 
\begin{equation}\label{eq2:L-f_vxvarphi}
\Delta_{(\HH,h)} \cdot f_{v\otimes \varphi} = f_{(-\pi(C_{(\HH,h)}) v)\otimes \varphi}. 
\end{equation}

Clearly, the trivial irreducible representation $1_G$ of $G$ contributes to $\Spec(\Delta_{(\HH,h)})$ with the eigenvalue $0$ (the smallest one) with multiplicity one. 
By proceeding in the same way as for \eqref{eq2:lambda1(G,g_A)}, one gets that the second (possible zero) smallest eigenvalue of $\Delta_{(\HH,h)}$ is given by 
\begin{equation}\label{eq2:lambda_1(G,H,g_I|H)}
\lambda_1(G/K,\HH,h) = \min \left\{ \lambda_{\min}(\pi(-C_{(\HH,h)})|_{V_\pi^K}): \pi\in\widehat G_K,\, \pi\not\simeq 1_G \right\}. 
\end{equation}

By Hörmander's theorem (\cite{Hormander67}), $\Delta_{(\HH,h)}$ is hypoelliptic when $\HH$ is bracket generating (i.e.\ the Lie algebra generated by $\fk\cup\HH$ is $\fg$, according to Definition~\ref{def2:bracket-generating}). 
In particular, $\Delta_{(\HH,h)}$ has a discrete spectrum since the inverse operator to $1+\Delta_{(\HH,h)}$ is compact.
We state this conclusion for future use. 

\begin{theorem}\label{thm2:Hormander-consequence}
If $\HH$ is an $\Ad(K)$-invariant bracket-generating subspace of $\fp$, $\lambda_1(G/K,\HH, g)<\infty$ for any $\Ad(K)$-invariant inner product $g$ on $\HH$. 
\end{theorem}

\section{Comparison}\label{sec:comparison}

In this short section we prove Theorem~\ref{thm1:comparison}. 
We use the assumptions and notations introduced in Remark~\ref{rem2:assumption+notation}.

\begin{lemma}\label{lem-comp:diam}
	Let $h$ be a $G$-invariant metric on $G/K$. 
	Then, for any left-invariant metric $g$ on $G$ satisfying that $g|_{\fp}=h$ and $g(\fp,\fk)=0$, we have that $\diam(G/K,h)\leq \diam(G,g)$. 
\end{lemma}

\begin{proof}
	Lemma~\ref{lem2:sing-diam-restriccionRiemanniano} gives $\diam(G,g) \geq \diam(G,g|_{\fp}^*)$ (see Notation~\ref{not2:symmetricbilinearform} for $g|_\fp^*$). 
	Since $g|_{\fp}^*= h|_{\fp}^*$, Lemma~\ref{lem2:diam(G/d)=diam(G/H)} (for $K=\{e\}$) tells us that $\diam(G/K,h) = \diam(G,h|_{\fp}^*) = \diam(G,g|_{\fp}^*)$, and the proof is complete. 
\end{proof}

\begin{lemma}\label{lem-comp:lambda1}
Assume that $K$ is connected. 
	Let $h$ be a $G$-invariant metric on $G/K$. 
	For $t>0$, we set $g_t=h|_{\fp}\oplus t^2g_0|_{\fk}$ in $\mathcal M(G):=\mathcal M(G,\{e\})$ (i.e.\ $g_t(X,Y)= h(X_\fp,Y_\fp)+ t^2g_0(X_\fk,Y_\fk)$ for all $X,Y\in \fg$, where the subindex $\fp$ and $\fk$ denotes the orhogonal projections to $\fp$ and $\fk$ respectively). 
	Then
	\begin{equation*}
		\lambda_1(G/K,h)= \lim_{t\to 0} \lambda_1(G, g_t). 
	\end{equation*}
\end{lemma}

\begin{proof}
Write $h=g_{\Phi}$ with $\Phi\in \Sym_K^+(\fp)$. 
Clearly, $g_t=g_{\Psi_t}$ with $\Psi_t= \Phi \oplus t^2\Id_{\fk} $. 
Thus 
\begin{equation*}
C_t:=C_{\Psi_t}
= t^{-2}C_{\fk} + C_\Phi
\end{equation*}
by \eqref{eq2:C_Phi}, 
where $C_\fk$ is the Casimir element of $\fk$ with respect to $\innerdots_0$, that is, $C_{\fk}= \sum_{i=1}^{\dim \fk} X_i^2$ for any orthonormal basis $\{X_1,\dots,X_{\dim\fk}\}$ of $\fk$ with respect to $\innerdots_0$. 

Let $(\pi,V_\pi)\in\widehat G$.
Let $\innerdots_\pi$ denote the Hermitian inner product on $V_\pi$ such that $\pi$ is unitary. 
We decompose $V_\pi=V_\pi^K\oplus (V_\pi^K)^\perp$ orthogonally with respect to $\innerdots_\pi$. 
It follows that the subspaces $V_\pi^K$ and $(V_\pi^K)^\perp$ are invariants by $\pi(C_\fk)$ and $\pi(C_\Phi)$, and furthermore, $\pi(C_\fk)|_{V_\pi^K}=0$. 

From \eqref{eq2:lambda1(G,g_A)}, we obtain that $\lambda_1(G,g_t)$ is given by the minimum in 
\begin{equation*}
\begin{aligned}
\left\{ \lambda_{\min}(\pi(-C_\Phi)|_{V_\pi^K}): \pi\in\widehat G_K,\, \pi\not\simeq 1_G \right\} 
\cup 
\left\{ \lambda_{\min}(\pi(-C_t)|_{(V_\pi^K)^\perp}): \pi\in\widehat G,\, \pi\not\simeq 1_G \right\} . 
\end{aligned}
\end{equation*}
It remains to show that the minimum is attained in the set at the left-hand side in the displayed formula for $t$ sufficiently small, since this number is precisely $\lambda_1(G/K,h)$ again by \eqref{eq2:lambda1(G,g_A)}. 

We fix $(\pi,V_\pi)\in\widehat G$. 
By restricting the action $\pi$ to $K$, $V_\pi$ decomposes as the sum of its isotypical components $V_\pi=\bigoplus_{i=0}^s U_{i}$.
More precisely, there are pairwise non-equivalent irreducible representations $\tau_0,\dots,\tau_s$ of $K$ such that $U_i$ is the sum of all $\pi(K)$-invariant subspaces equivalent to $\tau_i$. 
Note that $V_\pi^K$ is the isotypical component associated to the trivial representation $1_K$ of $K$. 
Thus, we can assume that $\tau_0=1_K$ by letting $U_0=0$ if $\pi\notin \widehat G_K$. 

One can check that $U_0,\dots,U_s$ are pairwise orthogonal with respect to $\innerdots_\pi$, thus $(V_\pi^K)^\perp= \bigoplus_{i=1}^s U_i$. 
Each $U_i$ is invariant by $\pi(C_\Phi)$ and $\pi(C_\fk)$. 
Moreover, since $C_\fk$ is a Casimir element, $\pi(-C_\fk)|_{U_i} = \lambda^{\tau_i}\Id_{U_i}$ for some scalar $\lambda^{\tau_i}$ for all $1\leq i\leq s$ (see Remark~\ref{rem2:C_I}). 
One has that $\lambda^{\tau_i}>0$ for all $1\leq i\leq s$ since $K$ is assumed connected. 
We deduce that there is a basis of eigenvectors $\{v_{i,j}: 1\leq i\leq s,\, 1\leq j\leq \dim U_i\}$ of $\pi(-C_t)|_{(V_\pi^K)^\perp}$ with $v_{i,j}\in U_i$ for all $j$. 
Since $\pi(-C_\fk)$ acts by an scalar on $U_i$, $v_{i,j}$ is also an eigenvector of $\pi(-C_\Phi)$, say $\pi(-C_\Phi)v_{i,j}= \nu_{i,j}v_{i,j}$. 
Hence, 
\begin{equation*}
\pi(-C_t)v_{i,j} = t^{-2}\pi(-C_\fk)v_{i,j} + \pi(-C_\Phi)v_{i,j} = (t^{-2}\lambda^{\tau_i}+\nu_{i,j})v_{i,j}.
\end{equation*}
Since $\lambda^{\tau_i}>0$ and $\nu_{i,j}\geq0$ for $i\geq 1$, the eigenvalue $t^{-2}\lambda^{\tau_i} +\nu_{i,j}$ can be as large as we need by taking $t$ sufficiently small. 
Moreover, since $\inf \{\lambda^{\tau}: \tau\in\widehat  K,\, \tau\not\simeq 1_K\}=\lambda_1(K,\widetilde g_0)>0$,
where $\widetilde g_0$ denotes the bi-invariant metric on $K$ induced by $\innerdots_0|_{\fk\times\fk}$, we conclude that $$\min \left\{ \lambda_{\min}(\pi(-C_t)|_{(V_\pi^K)^\perp}): \pi\in\widehat G,\, \pi\not\simeq 1_G \right\}$$ can be as large as we need, as requested. 
\end{proof}

\begin{proof}[Proof of Theorem~\ref{thm1:comparison}]
	By assumption, there is $C>0$ such that $\lambda_1(G,g)\diam(G,g)^2\leq C$ for all $g\in \mathcal M(G)$. 
	We fix $h \in \mathcal M(G,K)$. 
	For $t>0$, let $g_t$ be the left-invariant metric on $G$ as in the statement of Lemma~\ref{lem-comp:lambda1}. 
	From Lemmas~\ref{lem-comp:diam} and \ref{lem-comp:lambda1}, it follows that
	\begin{equation}
	\lambda_1(G/K,h)\diam(G/K,h)^2 \leq \limsup_{t\to 0} \lambda_1(G, g_t)\diam(G,g_t)^2\leq C,
	\end{equation}
	and the proof is complete. 
\end{proof}

\section{Diameter estimates} \label{sec:diam}
We will use the assumptions and notations introduced in Remark~\ref{rem2:assumption+notation} and Subsection~\ref{subsec:G-invmetrics}. 
In this section we estimate the diameter of the homogeneous Riemannian manifold $(G/K,g_\Phi)$ in terms of the eigenvalues of $\Phi$, for any $\Phi\in \Sym_K^+(\fp)$.

\subsection{Simple estimates for the diameter}

To motivate the diameter estimates of this section, we begin by discussing the simple estimates
\begin{equation}\label{eq3:diam-simple-estimates}
\frac{\diam(G/K,g_{\Id_\fp})}{\sigma_1(\Phi)} 
	\leq \diam(G/K,g_\Phi) \leq 
\frac{\diam(G/K,g_{\Id_\fp})}{\sigma_q(\Phi)}
\end{equation}
for $\Phi\in\Sym_K^+(\fp)$. 
We recall that, according to Notation~\ref{not2:diagonaldecomposition}, $\sigma_1(\Phi)^{-2}$ (resp.\ $\sigma_q(\Phi)^{-2}$) denotes the smallest (resp.\ largest) eigenvalue of $\Phi$.
This estimate will follow from the next result.

\begin{lemma}\label{lem3:diam-inequality}
Let $\Phi,\Psi$ be in $\Sym_K^+(\fp)$ satisfying that $\Phi \geq \Psi$. 
Then
$$
\diam(G/K,g_\Phi)\geq \diam(G/K,g_\Psi).
$$
\end{lemma}

\begin{proof}
By Lemma~\ref{lem2:A^tAleqB^tB}, $\Phi\geq \Psi$ forces to  $g_\Phi(X,X)\geq g_\Psi(X,X) $ for all $X\in\fp$. 
The proof follows by monotonicity (see e.g.\ \cite[Lem.~2.10]{Lauret-EGSconj}).
\end{proof}

We now prove \eqref{eq3:diam-simple-estimates}.
Let $\fp=\fp_1\oplus\dots\oplus \fp_q$ be any diagonal decomposition of $\Phi$ (see Notation~\ref{not2:diagonaldecomposition}). 
Then
\begin{equation}\label{eq2:A^tAprimero}
\Phi=
\bigoplus_{i}
\sigma_{i}(\Phi)^{-2}\,\Id_{\fp_{i}} 
\leq 
\bigoplus_{i} \sigma_q(\Phi)^{-2} \,\Id_{\fp_{i}} =\sigma_q(\Phi)^{-2}\, \Id_{\fp},
\end{equation}
since $\sigma_q(\Phi)^{-2}\geq \sigma_i(\Phi)^{-2}$ for all $i$.
Lemma~\ref{lem3:diam-inequality} gives
$
\diam(G/K,g_{\Phi}) \leq \diam (G/K,g_{\sigma_q(\Phi)^{-2}\Id_\fp}) 
=\diam (G/K,\sigma_q(\Phi)^{-2} g_{\Id_\fp}) , 
$
and consequently the inequality at the right-hand side of \eqref{eq3:diam-simple-estimates} follows since, for any arbitrary Riemannian manifold, one has that
\begin{equation}\label{eq3:diam_scale}
\begin{aligned}
\diam(M,t \,g) &= \sqrt{t}\diam(M,g)
\quad\text{for all }t>0.
\end{aligned}
\end{equation}
The other estimate follows analogously by using $\Phi\geq \sigma_1(\Phi)^{-2} \, \Id_\fp$.

\begin{notation}\label{not2:HH_P^kCC_P^k}
For $\Phi\in\Sym_K^+(\fp)$, a diagonal decomposition $\fp=\fp_1\oplus\dots\oplus \fp_q$ of $\Phi$, and any index $1\leq k\leq q$, we set 
\begin{equation}
\HH_{\Phi,\{\fp_{i}\},k} = \bigoplus_{i=1}^k \fp_{i}
\quad\text{and}\quad
\CC_{\Phi,\{\fp_{i}\},k} = \bigoplus_{i=k}^{q} \fp_{i}.
\end{equation}
(We recall from Notation~\ref{not2:diagonaldecomposition} that the numbers $\sigma_1(\Phi),\dots,\sigma_q(\Phi)$ are determined by $\Phi=\bigoplus_{i} \sigma_{i}(\Phi)^{-2}\, \Id_{\fp_{i}} $ and $\sigma_1(\Phi)\geq\dots\geq \sigma_q(\Phi)$.)
We will abbreviate them by $\HH_{k}$ and $\CC_{k}$ when $\Phi$ and $\{\fp_{i}\}$ are clear in the context. 
\end{notation}

The next result will use the corresponding sub-Riemannian manifold $(G/K,\HH_k,g_0|_{\HH_k})$ and the singular Riemannian manifold $(G/K,g_0|_{\CC_k}^*)$ introduced in Subsection~\ref{subsec:diam-non-Riemannian}.

\begin{proposition}\label{prop3:diam-k}
Let $\Phi\in \Sym_K^+(\fp)$ and let $\fp=\fp_1\oplus\dots\oplus \fp_q$ be a diagonal decomposition of $\Phi$ (see Notation~\ref{not2:diagonaldecomposition}). 
For any index $1\leq k\leq q$, we have that
\begin{equation}\label{eq3:diamA}
\frac{ \diam(G/K,g_0|_{\CC_{k}}^* )} {\sigma_k(\Phi)} 
\leq \diam(G/K,g_\Phi) \leq 
\frac{\diam(G/K,\HH_{k},g_0|_{\HH_{k}} )}{\sigma_k(\Phi)}.
\end{equation}
\end{proposition}

\begin{proof}
We abbreviate $\sigma_i=\sigma_i(\Phi)$ for all $i$. 
We set 
\begin{equation*}
\Psi_1=\left(\bigoplus_{i<k} \tfrac{\sigma_k^2}{\sigma_{i}^2} \, \Id_{\fp_{i,j}}\right) 
\bigoplus \;\Id_{\CC_k}. 
\end{equation*}
Similarly as in \eqref{eq2:A^tAprimero}, one can easily see that $\Phi\geq \sigma_k^{-2}\Psi_1$.
Lemma~\ref{lem3:diam-inequality} and \eqref{eq3:diam_scale} imply that 
$
\diam(G/K,g_\Phi)
\geq \diam(G/K,g_{\sigma_k^{-2} \Psi_1}) 
= \sigma_k^{-1}\diam(G/K,g_{\Psi_1}).
$ 
Hence, the inequality at the left-hand side in \eqref{eq3:diamA} follows since $\diam(G/K,g_{\Psi_1})\geq \diam(G/K,g_0|_{\CC_{k}}^* )$ by Lemma~\ref{lem2:sing-diam-restriccionRiemanniano}. 

We now establish the inequality at the right in \eqref{eq3:diamA}. 
Similarly as above, by setting 
\begin{equation*}
\Psi_2=\left(\bigoplus_{i>k} \tfrac{\sigma_{i}^2}{\sigma_k^2} \, \Id_{\fp_{i,j}}\right) 
\bigoplus \; \Id_{\HH_k},
\end{equation*}
one has $\Phi\leq \sigma_k^{-2}\, \Psi_2$, thus $\diam(G/K,g_\Phi) \leq \sigma_k^{-1}\diam(G/K,g_{\Psi_2})$.
The assertion follows since $\diam(G/K,g_{\Psi_2})\leq \diam(G/K,\HH_{k},g_0|_{\HH_{k}} )$ by Lemma~\ref{lem2:sub-diam-restriccionRiemanniano}. 
\end{proof}

\begin{remark}
Some words of caution about \eqref{eq3:diamA} are necessary at this point. 
Unlike in \eqref{eq3:diam-simple-estimates}, the coefficients in the extremes depend on $\Phi$ (more precisely on $\{\fp_{i}\}$). 
Moreover, the inequality at the left-hand (resp.\ right-hand) side is useless when $\diam(G/K,g_0|_{\CC_{k}}^*)=0$ (resp.\ $\diam(G/K,\HH_{k}, g_0|_{\HH_{k}} )=\infty $).
\end{remark}

\section{Eigenvalue estimates} \label{sec:eigenvalues}
We continue using the assumptions and notations introduced in Remark~\ref{rem2:assumption+notation} and Subsection~\ref{subsec:G-invmetrics}. 
In this section we estimate the smallest positive eigenvalue of the Laplace--Beltrami operator associated to $(G/K,g_\Phi)$ for $\Phi\in \Sym_K^+(\fp)$, in terms of the eigenvalues of $\Phi$.
We will proceed analogously to the previous section.

\subsection{Simple estimates for the first eigenvalue}
We have seen in \eqref{eq2:A^tAprimero} that $\sigma_1(\Phi)^2\Id_\fp\leq \Phi \leq \sigma_q(\Phi)^2\Id_\fp$ for all $\Phi\in \Sym_K^+(\fp)$.
Recall from Notation~\ref{not2:diagonaldecomposition} that $\sigma_1(\Phi)^{-2}$ and $\sigma_q(\Phi)^{-2}$ stand for the smallest and largest eigenvalues of $\Phi$ respectively. 
The estimates
\begin{equation}\label{eq4:lambda1-simple-estimates}
\lambda_1(G/K,g_{\Id_\fp})\, \sigma_q(\Phi)^2
\leq \lambda_1(G/K,g_\Phi)  \leq 
\lambda_1(G/K,g_{\Id_\fp})\, \sigma_1(\Phi)^2
\end{equation}
follow immediately form the next result.

\begin{lemma} \label{lem4:lambda1-inequality}
Let $\Phi,\Psi$ be in $\Sym_K^+(\fp)$ satisfying that $\Phi \geq \Psi$. 
Then $\pi(-C_\Phi)|_{V_\pi^K}\leq \pi(-C_\Psi)|_{V_\pi^K}$ for every finite dimensional unitary representation $\pi$ of $G$. 
Moreover, 
$$
\lambda_1(G/K,g_\Phi)\leq \lambda_1(G/K,g_\Psi).
$$
\end{lemma}

\begin{proof}
Let $\{Y_1,\dots,Y_n\}$ be any basis of $\fp$. 
We set $S=(\inner{Y_i}{Y_j}_\Phi)_{i,j}$, $S'=(\inner{Y_i}{Y_j}_\Psi)_{i,j}$, $T=S^{-1}=(t_{i,j})_{i,j}$, and $T'=(S')^{-1}=(t_{i,j}')_{i,j}$, which are all $n\times n$ positive definite symmetric matrices satisfying $S\geq S'$ and $T\leq T'$ since $\Phi\geq \Psi$. 
Write $\widetilde T=T'-T=(\widetilde t_{i,j})_{i,j}$. 
According to \eqref{eq2:C_Phi}, we have that 
\begin{equation}\label{eq4:C_Psi=C_Phi+tildeC}
C_{\Psi} 
= \sum_{i,j}t_{i,j}'\; Y_i\, Y_j
= \sum_{i,j}\, (t_{i,j}+\widetilde t_{i,j}) \; Y_i\, Y_j 
=C_\Phi+\widetilde C,
\end{equation}
where $\widetilde C= \sum_{i,j} \widetilde t_{i,j}\; Y_i\, Y_j$. 
Since $\widetilde T$ is a positive semi-definite symmetric matrix, there are $P=(p_{i,j})_{i,j} \in\Ot(n)$ and $d_1,\dots,d_n\in\R$ such that $\widetilde T=PD^2P^t$, where $D=\diag(d_1,\dots,d_n)$. 
Let $\{Z_1,\dots,Z_n\}$ be the basis of $\fp$ determined by $Y_i=\sum_{j} p_{i,j}Z_j$ for all $i$.
Hence
\begin{equation*}
	\begin{aligned}
		\widetilde C 
		&= \sum_{i,j} \widetilde t_{i,j} (\sum_k p_{i,k}Z_k)(\sum_l p_{j,l}Z_l)
		= \sum_{k,l} (\sum_{i,j} p_{i,k}\widetilde t_{i,j}p_{i,l})\;   Z_k\,Z_l
		\\
		&= \sum_{k,l} (P^t\widetilde TP)_{k,l}\;   Z_k\,Z_l
		= \sum_{k} d_{k}^2\;   Z_k^2.
	\end{aligned}
\end{equation*}

Let $(\pi,V_\pi)$ be any finite dimensional unitary representation of $G$.
By \eqref{eq4:C_Psi=C_Phi+tildeC}, we have that 
\begin{equation}
\pi(-C_\Psi)|_{V_\pi^K}=\pi(-C_\Phi)|_{V_\pi^K}+ \pi(-\widetilde C)|_{V_\pi^K}\geq \pi(-C_\Phi)|_{V_\pi^K}
\end{equation}
since $\pi(-\widetilde C)|_{V_\pi^K}=\sum_k d_k^2\pi(-X_k^2)|_{V_\pi^K}\geq 0$ because $\pi(-X^2) \geq0$ for any $X\in\fg$.

We now show the second assertion. 
For $\pi\in \widehat G$, since $\pi(-C_\Phi)|_{V_\pi^K} \leq \pi(-C_\Psi)|_{V_\pi^K}$, we have that $\lambda_{\min}(\pi(-C_\Phi)|_{V_\pi^K}) \leq  \lambda_{\min}(\pi(-C_\Psi)|_{V_\pi^K})$.
Hence, \eqref{eq2:lambda1(G,g_A)} immediately implies that $\lambda_1(G/K,g_\Phi)\leq \lambda_1(G/K,g_\Psi)$. 
\end{proof}

\begin{notation}
Let $\Phi\in\Sym_K^+(\fp)$. 
For any $\Ad(K)$-invariant subspace $\fq$ of $\fp$ we set 
\begin{equation}
C_{\fq}^\infty(G/K) =\{f\in C^\infty(G/K): X\cdot f=0 \;\text{for all } X\in\fq\}.
\end{equation} 
By \eqref{eq2:PeterWeyl}, the closure of $C_{\fq}^\infty(G/K)$ in the Hilbert space $L^2(G/K)$ is given by
\begin{equation}\label{eq4:closureC_Pk^infty}
\op{closure}(C_{\fq}^\infty(G/K)) = \bigoplus_{\pi\in\widehat G} (V_\pi^K)^{\fq} \otimes V_\pi^*,
\end{equation}
where $(V_\pi^K)^{\fq}=\{ v\in V_\pi^K: \pi(X)\cdot v=0\text{ for all }X\in\fq\}$. 
In particular, by Remark~\ref{rem2:C_I}, the Laplace--Beltrami operator $\Delta_{\Id_\fp}$ of $(G/K,g_{\Id_\fp})$ preserves $C_{\fq}^\infty(G/K)$. 
Whenever $C_{\fq}^\infty(G/K)$ has dimension strictly greater than one, we denote by  $\lambda_1(\Delta_{\Id_\fp}|_{C_{\fq}^\infty(G/K)})$ the smallest positive eigenvalue of $\Delta_{\Id_\fp}|_{C_{\fq}^\infty(G/K)}$.
Otherwise, when $C_{\fq}^\infty(G/K)$ contains only constant functions on $G/K$ (e.g.\ if $\fq$ is bracket generating because $\dim (V_\pi^K)^{\fq}=0$ for all $\pi\in\widehat G_K$), we set $\lambda_1(\Delta_{\Id_\fp}|_{C_{\fq}^\infty(G/K)})=\infty$ by convention. 
\end{notation}

\begin{remark}
From \eqref{eq4:closureC_Pk^infty}, we know that  $f_{v\otimes \varphi}$ is an eigenfunction of $\Delta_{\Id_\fp}|_{C_{\fq}^\infty(G/K)}$ for every $\varphi\in V_\pi^*$ and every non-zero eigenvector $v\in (V_\pi^K)^{\fq}$ of $\pi(-C_{\Id_\fp})$. 
Consequently, 
\begin{equation}\label{eq4:spec_I|C_Pk}
\Spec (\Delta_{\Id_\fp}|_{C_{\fq}^\infty(G/K)}) = \bigcup_{\pi\in\widehat G_K} \big\{\!\big\{ \rule{-4mm}{0mm}
\underbrace{\lambda^{\pi},\dots, \lambda^{\pi}}_{(d_\pi\dim (V_\pi^K)^{\fq})\text{-times}} 
\rule{-4mm}{0mm}\big\}\!\big\}.
\end{equation}
Recall from Remark~\ref{rem2:C_I} that $\lambda^\pi$ is determined by $\pi(-C_{\Id_\fp})|_{V_\pi^K} =  \lambda^\pi\, \Id_{V_\pi^K}$, for any $\pi \in \widehat G_K$.

Suppose there is a closed subgroup $H$ of $G$ with Lie algebra $\fh$ satisfying that $\fk\oplus \fq\subset \fh\neq\fg$. 
Clearly, $V_\pi^H\subset (V_\pi^K)^\fq$ for all $\pi\in\widehat G$. 
Thus $C^\infty(G/H) \subset  C_\fq^\infty(G/K)$ and consequently 
\begin{equation}\label{eq:lambda1(singular)<infty}
\lambda_1(\Delta_{\Id_\fp}|_{C_{\fq}^\infty(G/K)})\leq \lambda_1(G/H,g_{\Id_{\fh^\perp}})<\infty, 
\end{equation}
where $\fh^\perp$ denotes the orthogonal complement of $\fh$ in $\fg$ with respect to $\innerdots_0$. 
\end{remark}

We recall from Notation~\ref{not2:HH_P^kCC_P^k} that, once $\Phi\in\Sym_K^+(\fp)$ and a diagonal decomposition $\fp=\fp_1\oplus\dots\oplus\fp_q$ of $\Phi$ (see Notation~\ref{not2:diagonaldecomposition}) are fixed, we have 
\begin{equation}
\HH_{k} = \bigoplus_{i\leq k} \fp_{i}
\quad\text{and}\quad
\CC_{k} = \bigoplus_{i\geq k} \fp_{i},
\quad\text{and we set}\quad 
\mathcal F_k=\CC_{k}^\perp = \bigoplus_{i<k} \fp_{i},
\end{equation}
where the numbers $\sigma_1(\Phi)\geq\dots\geq \sigma_q(\Phi)$ satisfy $\Phi=\bigoplus_i \frac{1}{\sigma_i(\Phi)^2}\, \Id_{\fp_i}$.

\begin{proposition}\label{prop4:lambda1-k}
Let $\Phi\in\Sym_K^+(\fp)$ and let $\fp=\fp_1\oplus\dots\oplus\fp_q$ be a diagonal decomposition of $\Phi$ (see Notation~\ref{not2:diagonaldecomposition}). 
For any index $1\leq k\leq q$, we have that 
\begin{equation}\label{eq3:lambda_1(A)}
\lambda_1(G/K,\HH_{k},g_0|_{\HH_{k}}) \; \sigma_k(\Phi)^2
\leq \lambda_1(G/K,g_\Phi) \leq 
\lambda_1(\Delta_{\Id_{\fp}}|_{C_{\FF_{k}}^\infty(G/K)})\; \sigma_k(\Phi)^2.
\end{equation}
\end{proposition}

\begin{proof}
We abbreviate $\sigma_i=\sigma_i(\Phi)$ for all $i$. 
We set
\begin{align*}
\Psi_1&=\Big(\bigoplus_{i<k} \tfrac{\sigma_k^2}{\sigma_{i}^2} \, \Id_{\fp_{i}}\Big) 
	\bigoplus \;\Id_{\CC_k},
&
\Psi_2&=\Big(\bigoplus_{i>k} \tfrac{\sigma_{i}^2}{\sigma_k^2} \, \Id_{\fp_{i}}\Big) 
	\bigoplus \; \Id_{\HH_k}.
\end{align*}
It follows immediately that $\sigma_k^{-2} \, \Psi_1 \leq \Phi\leq \sigma_k^{-2} \,\Psi_2$.
By Lemma~\ref{lem4:lambda1-inequality}, we obtain that
\begin{equation*}
\sigma_k^2\, \lambda_1(G/K,g_{\Psi_2}) 
= \lambda_1(G/K,g_{\sigma_k^{-2} \Psi_2}) 
\leq \lambda_1(G/K,g_\Phi)\leq 
\lambda_1(G/K,g_{\sigma_k^{-2} \Psi_1}) =
\sigma_k^2\, \lambda_1(G/K,g_{\Psi_1}). 
\end{equation*}
It remains to show that 
\begin{align}\label{eq4:claim-prop-lambda_1k}
\lambda_1(G/K,\HH_{k},g_0|_{\HH_{k}})
\leq \lambda_1(G/K,g_{\Psi_2})
\quad\text{ and }\quad
\lambda_1(G/K,g_{\Psi_1})\leq \lambda_1(\Delta_{\Id_{\fp}}|_{C_{\FF_k}^\infty(G/K)}).
\end{align}

Let $\{X_{i,l}\}$ be an orthonormal basis of $\fp$ with respect to $\innerdots_0$ that respects the diagonal decomposition of $\Phi$ (see Remark~\ref{rem2:basis}), that is, $\fp_{i}=\Span_\R\{X_{i,l}:1\leq l\leq \dim \fp_i\}$ for all $i$. 
From \eqref{eq2:C_Phi-diagonal}, we have that 
\begin{align*}
C_{\Psi_2} = 
\sum_{i\leq k} \sum_{l=1}^{\dim\fp_i} X_{i,l}^2 +
\sum_{i>k} \tfrac{\sigma_{i}^2}{\sigma_k^2} \sum_{l=1}^{\dim\fp_i} X_{i,l}^2  .
\end{align*}
Let $(\pi,V_\pi) \in\widehat G_K$. 
Since $\pi(-X^2)\geq0$ for every $X\in\fg$, we conclude that 
\begin{align*}
\pi(-C_{(\HH_k,g_0|_{\HH_k})})|_{V_\pi^K} 
= \sum_{i\leq k} \sum_{l=1}^{\dim\fp_i} \pi(-X_{i,l}^2)|_{V_\pi^K}
\leq \pi(-C_{\Psi_2})|_{V_\pi^K}.
\end{align*}
Consequently, $\lambda_{\min}(\pi(-C_{(\HH_k,g_0|_{\HH_k})})|_{V_\pi^K}) \leq \lambda_{\min}(\pi(-C_{\Psi_2})|_{V_\pi^K})$, thus the first inequality in \eqref{eq4:claim-prop-lambda_1k} follows by \eqref{eq2:lambda1(G,g_A)} and \eqref{eq2:lambda_1(G,H,g_I|H)}.

We now establish the inequality at the right-hand side in \eqref{eq4:claim-prop-lambda_1k}. 
We assume that the dimension of $C^\infty_{\FF_k}(G/K)$ is greater than one, otherwise the assertion is trivial. 
From \eqref{eq4:spec_I|C_Pk}, it suffices to show that $\lambda_1(G/K,g_{\Psi_1})\leq \lambda^\pi$ for all $\pi\in\widehat G$ satisfying that $\dim (V_\pi^K)^{\FF_{k}}>0$. 
	
Let $\pi_0\in\widehat G$ satisfying $(V_\pi^K)^{\FF_{k}}\neq0$ and let $v_0\in (V_\pi^K)^{\FF_{k}}$ with $\langle v_0, v_0\rangle_{\pi_0}=1$, where $\innerdots_{\pi_0}$ denotes the Hermitian inner product on $V_\pi$. 
From \eqref{eq2:C_Phi-diagonal}, we have that 
\begin{align*}
C_{\Psi_1} 
= \sum_{i\geq k} \sum_{l=1}^{\dim\fp_i} X_{i,l}^2 +
\sum_{i<k} \tfrac{\sigma_k^2}{\sigma_{i}^2} \sum_{l=1}^{\dim\fp_i} X_{i,l}^2 .
\end{align*}
Note that $\pi_0(X_{i,l})v_0=0$ for all $i,l$ satisfying $i<k$. 
Hence
\begin{multline*}
\lambda_1(G/K,g_{\Psi_1})
	\leq \; \lambda_{\min}(\pi_0(-C_{\Psi_1})|_{V_{\pi_0}^K}) 
	= \min_{v\in V_{\pi_0}^K:\, \langle v,v\rangle_{\pi_0}=1} \langle \pi_0(-C_{\Psi_1})v,v \rangle_{\pi_0} \\
	\leq \langle \pi_0(-C_{\Psi_1})v_0,v_0\rangle_{\pi_0}
	=\sum_{i\geq k} \sum_{l=1}^{\dim\fp_i} \langle \pi_0(-X_{i,k}^2)v_0,v_0\rangle_{\pi_0}
\\
	=\sum_{i}\sum_{l=1}^{\dim\fp_i} \langle \pi_0(-X_{i,k}^2)v_0,v_0\rangle_{\pi_0}
	=\langle \pi_0(-C_{\Id_\fp})v_0,v_0\rangle_{\pi_0} = \lambda^{\pi_0},
\end{multline*}
and the proof is complete. 
\end{proof}

\section{Proof of the main results} \label{sec:upperbounds}

In this section, we will prove Theorems~\ref{thm1:multfreeisotropy} and \ref{thm1:spheres} by the diameter estimates from Section~\ref{sec:diam} and the eigenvalue estimates from Section~\ref{sec:eigenvalues}. 
Let the assumptions and notations be as in Remark~\ref{rem2:assumption+notation} and Subsection~\ref{subsec:G-invmetrics}. 
In particular, $G$ is a compact connected Lie group, $K$ is a closed subgroup of $G$ and, at the Lie algebra level, we have the orthogonal decomposition $\fg=\fk\oplus\fp$ with respect to a fixed $\Ad(G)$-invariant inner product $\innerdots_0$. 
Furthermore, the space $\mathcal M(G,K)$ of $G$-invariant metrics on $G/K$ is in correspondence with the set $\Sym_K^+(\fp)$ of $\Ad(K)$-invariant positive definite symmetric endomorphisms on $\fp$.

From the simple estimates \eqref{eq3:diam-simple-estimates} and \eqref{eq4:lambda1-simple-estimates}, we obtain 
\begin{equation}\label{eq5:simple-estimates}
\frac{\sigma_q(\Phi)^2}{\sigma_1(\Phi)^2} \, c_0 \leq
\lambda_1(G/K,g_{\Phi})\, \diam(G/K,g_{\Phi})^2 
\leq \frac{\sigma_1(\Phi)^2}{\sigma_q(\Phi)^2} \, c_0
\end{equation}
for every $\Phi\in \Sym_K^+(\fp)$, 
where $c_0=\lambda_1(G/K,g_{\Id_\fp}) \diam(G/K,g_{\Id_\fp})^2$. 
In general, these estimates are not useful since $\sigma_1(\Phi)/\sigma_q(\Phi)$ is not bounded by any positive real number for $\Phi\in \Sym_K^+(\fp)$. 
The only exception occurs when $G/K$ is isotropy irreducible (i.e.\ the case $r=1$ and $q_1=1$ in \eqref{eq2:isotropydecomposition}), since we have $q=1$ thus $\sigma_1(\Phi)=\sigma_q(\Phi)$. 
However, in this case, there is a unique $G$-invariant metric on $G/K$ up to positive scaling, thus the functional $g\mapsto \lambda_1(G/K,g) \diam(G/K,g)^2$ is indeed constant for $g\in\mathcal M(G,K)$.

\begin{proposition}\label{prop6:ordered-estimates}
Let $G$ be a compact connected Lie group and let $K$ be a closed subgroup of $G$ such that $G/K$ is connected. 
We fix a decomposition $\fp=\fp_1\oplus\dots\oplus \fp_q$ in $\Ad(K)$-invariant subspaces. 
Let $k$ be the smallest index satisfying that $\bigoplus_{i=1}^k \fp_i$ is bracket generating.
We set 
\begin{align*}
\HH&=\bigoplus_{i=1}^k \fp_i,&
\FF&=\bigoplus_{i=1}^{k-1} \fp_i, &
\CC&=\bigoplus_{i=k}^{q} \fp_i.
\end{align*}
We assume that there is a proper subalgebra $\fh$ of $\fg$ containing $\fk$ and $\FF$ such that the associated connected subgroup $H$ of $G$ is closed in $G$. 
Then, 
\begin{equation}\label{eq5:estimates}
\begin{aligned}
\lambda_1(G/K,g_\Phi)\diam(G/K,g_\Phi)^2  &\geq \lambda_1(G/K,\HH,g_0|_{\HH}) \diam(G/K,g_0|_{\CC}^*)^2>0,\\
\lambda_1(G/K,g_\Phi)\diam(G/K,g_\Phi)^2  &\leq 
\lambda_1(\Delta_{\Id_{\fp}}|_{C_{\FF}^\infty(G/K)}) \diam(G/K,\HH,g_0|_{\HH})^2< \infty,
\end{aligned}
\end{equation}
for all $\Phi\in\Sym_K^+(\fp)$ satisfying that $\Phi|_{\fp_i} =\sigma_i(\Phi)^{-2}\Id_{\fp_i}$ for all $i$, with $\sigma_1(\Phi)\geq\dots\geq \sigma_q(\Phi)>0$.
\end{proposition}

\begin{proof}
Note that if $\Phi\in\Sym_K^+(\fp)$ satisfies that $\Phi|_{\fp_i} =\sigma_i(\Phi)^{-2}\Id_{\fp_i}$ for all $i$ and $\sigma_1(\Phi)\geq\dots\geq \sigma_q(\Phi)$, then $\fp=\fp_1\oplus\dots\oplus \fp_q$ is a diagonal decomposition of $\Phi$ according to Notation~\ref{not2:diagonaldecomposition}. 
Thus, the two inequalities at the left in each row of \eqref{eq5:estimates} follow immediately from Propositions~\ref{prop3:diam-k} and \ref{prop4:lambda1-k}. 
It remains to show that the bounds are positive real numbers. 

Since $\HH$ is bracket generating, 
Theorems~\ref{thm2:Chow} and \ref{thm2:Hormander-consequence} give $\diam(G/K,\HH,g_0|_{\HH})<\infty$ and $\lambda_1(G/K,\HH,g_0|_{\HH})>0$ respectively. 
By using the assumption of the existence of the subalgebra $\fh$, we obtain that $\diam(G/K,g_0|_{\CC}^*)>0$ by Proposition~\ref{prop2:diam-pseudo-Hclosed} and 
\begin{equation*}
	\lambda_1(\Delta_{\Id_{\fp}}|_{C_{\FF}^\infty(G/K)})
	\leq 
	\lambda_1(G/H,g_{\Id_{\fh^\perp}})<\infty
\end{equation*}
by \eqref{eq:lambda1(singular)<infty}, and the proof is complete. 
\end{proof}

\begin{theorem}\label{thm5:EGSpartial}
Let $G$ be a semisimple compact connected Lie group and let $K$ be a closed subgroup of $G$ such that $G/K$ is connected. 
We fix a decomposition $\fp=\fp_1\oplus\dots\oplus \fp_q$ in $\Ad(K)$-invariant subspaces. 
Then, there exist $C>0$ such that 
\begin{equation}\label{eq5:main-estimate}
\lambda_1(G/K,g_\Phi)\diam(G/K,g_\Phi)^2\leq C
\end{equation}
for all $\Phi\in \Sym_K^+(\fp)$ preserving 
$
\fp=\bigoplus_{i=1}^{r} \fp_i, 
$ 
that is, $\Phi(\fp_{i}) \subset \fp_{i}$ for all $i$. \end{theorem}

\begin{proof}
For $\xi$ in $\mathbb S_q$, the set of permutations of $\{1,\dots,q\}$, we set $\fp_{\xi,i}=\fp_{\xi(i)}$ for all $i$ and let $S_\xi$ denote the set of $\Phi\in\Sym_K^+(\fp)$ satisfying that $\Phi= \bigoplus_i \sigma_i^{-2}\Id_{\fp_{\xi,i}}$ with $\sigma_1\geq\dots\geq \sigma_q>0$.

Fix $\xi\in\mathbb S_q$. 
We next check the assumptions of Proposition~\ref{prop6:ordered-estimates} applied to the decomposition $\fp=\fp_{\xi,1}\oplus\dots\oplus\fp_{\xi,q}$.
Let $k$ be the smallest index satisfying that $\HH:=\bigoplus_{i=1}^k \fp_{\xi,i}$ is bracket generating. 
Set $\FF=\bigoplus_{i=1}^{k-1} \fp_{\xi,i}$.
Let $\fh_0$ be the subalgebra of $\fg$ generated by $\fk\cup \FF$, which is proper since $\FF$ is not bracket generating. 
Let $H_0$ be the connected subgroup of $G$ with Lie algebra $\fh_0$. 
Since a compact semisimple Lie group $\fg$ does not admit any dense proper subgroup (see \cite{MaciasVirgos}), we have that $\bar H_0\neq G$. 
We conclude that the Lie algebra $\fh$ of $\bar H_0$ satisfies the required assumption in Proposition~\ref{prop6:ordered-estimates}. 

Now, Proposition~\ref{prop6:ordered-estimates} implies that there is $C_\xi>0$ such that $\lambda_1(G/K,g_\Phi)\diam(G/K,g_\Phi)^2\leq C_\xi$ for all $\Phi\in S_\xi$. 
Hence, $C:=\max_{\xi\in\mathbb S_q} C_\xi$ satisfies \eqref{eq5:main-estimate} since the set of $\Phi\in\Sym_K^+(\fp)$ preserving the decomposition $\fp=\bigoplus_{i=1}^{r} \fp_{\xi,i}$ coincides with $\cup_{\xi\in\mathbb S_q}S_\xi$.  
\end{proof}

We are now in position to prove the main theorems. 

\begin{proof}[Proof of Theorem~\ref{thm1:multfreeisotropy}]
Since the isotropy representation is multiplicity free, we have that the decomposition $\fp=\fp_1\oplus \dots\oplus \fp_r$ in isotropic components is the unique (up to order) decomposition of $\fp$ in irreducible $K$-sub-modules, and furthermore, every $\Phi\in\Sym_K^+(\fp)$ preserves it (see Remark~\ref{ex2:multiplicityfree}). 
The assertion thus follows immediately from Theorem~\ref{thm5:EGSpartial}. 
\end{proof}

\begin{remark}\label{rem5:dense}
	The semisimplicity assumption on $G$ in Theorem~\ref{thm5:EGSpartial} cannot be omitted, but it can be replaced by (involved) weaker hypotheses such as in the statement of Proposition~\ref{prop6:ordered-estimates}. 
	It is not clear to the author whether the semisimplicity assumption on $G$ in Theorem~\ref{thm1:multfreeisotropy} can be omitted. 
\end{remark}

\begin{proof}[Proof of Theorem~\ref{thm1:spheres}]
According to \cite{Ziller82}, any homogeneous metric on a CROSS (compact rank one symmetric space) $X$ is isometric to some $G$-invariant metric on $G/K$ for some $(G,K)$ in the following table:
\begin{equation*}
\begin{array}{cllccc}
\;\;\text{row}\;\; &G&K& G/K&\text{cond.}&\text{isotropy rep.} 
\\ \hline
\rule{0pt}{12pt}
1&\SO(2n+1) & \SO(2n) & S^{2n} & n\geq1 & W_1 
\\
2&\SU(2) & \{e\} & S^3 & & 3W_0
\\
3&\SU(2n+1) & \SU(2n) & S^{4n+1} & n\geq2 & W_0\oplus W_1
\\
4&\Sp(n+1) & \Sp(n) & S^{4n+3} & n\geq 1 & 3W_0\oplus W_1
\\
5&\Spin(9) & \Spin(7) & S^{15} &  & W_1\oplus W_2
\\
6&\SU(2n+1) & \mathrm{S}(\Ut(2n)\times\Ut(1)) & P^{2n}(\C) &n\geq 1& W_1
\\
7&\Sp(n+1) & \Sp(n)\times\Ut(1) & P^{2n+1}(\C) &n\geq 1& W_1\oplus W_2
\\
8&\Sp(n+1) & \Sp(n)\times\Sp(1) & P^n(\mathbb H) &n\geq 1& W_1
\\
9&\textrm{F}_4 & \Spin(9) & P^2(\mathbb O) && W_1.
\end{array}
\end{equation*}
In each row of the table, $W_0$ denotes the trivial representation of $K$ and $W_1,W_2$ are some non-trivial and non-equivalent irreducible real representations of $K$. 

Since there are only finitely many realizations $X=G/K$ for each dimension $d$, it is sufficient to show the assertion for each individual family. 
The cases in rows 1, 3, and 5--9 follow from Theorem~\ref{thm1:multfreeisotropy} since their isotropy representations are multiplicity free.
The case in the second row was shown in \cite{EldredgeGordinaSaloff18}. 
It only remains the case in the fourth row. 

We set 
\begin{equation*}
\begin{aligned}
G &=\Sp(n+1) =\{A\in \GL(n+1,\mathbb H): A^*A=I\}, \\
K &= \left\{\begin{pmatrix} A\\ &1\end{pmatrix}\in G: A\in\Sp(n)\right\} \simeq \Sp(n),\\
\fp_1&=\left\{\begin{pmatrix} &v \\ -v^*\end{pmatrix}\in \fg: v\in \mathbb H^n \right\} ,
\end{aligned}
\end{equation*}
$X_1=\diag(0,\dots,0,\mathrm{i})$, $X_2=\diag(0,\dots,0,\mathrm{j})$, and $X_3=\diag(0,\dots,0,\mathrm{k})$. 
One has that the isotropy representation decomposes as $\fp=\fp_0\oplus \fp_1$, where the action on   $\fp_0:=\Span_\R\{X_1,X_2,X_3\}$ is trivial and the action on $\fp_1$ is equivalent to the standard representation. 

Ziller~\cite{Ziller82} showed that every $G$-invariant metric on $G/K$ is isometric to $g_\Phi$ with $\Phi\in\Sym_K^+(\fp)$ preserving the decomposition $\fp=\fa_1\oplus \fa_2\oplus \fa_3\oplus\fp_1$, where $\fa_i=\Span_\R\{X_i\}$, that is, $\Phi(\fp_1)=\fp_1$ and $\Phi(\fa_i)=\fa_i$ for all $i$. 
The proof follows by applying Theorem~\ref{thm5:EGSpartial} to this particular decomposition. 
\end{proof}

\begin{remark}
For $G/K=\Sp(n+1)/\Sp(n)\simeq S^{4n+3}$ ($n\geq1$), it has been recently determined in \cite{BLPhomospheres} an explicit expression for $\lambda_1(S^{4n+3},g)$ for every $G$-invariant metric $g$ on $S^{4n+3}$.
This might help to obtain explicit constants $C_1,C_2>0$ satisfying that $C_1\leq \lambda_1(S^{4n+3},g) \, \diam(S^{4n+3},g)^2\leq C_2$ for all $g\in\mathcal M(G,K)$, analogous as those obtained in \cite{Lauret-SpecSU(2)} for $S^3$. 
\end{remark}

\bibliographystyle{plain}

\end{document}